\documentclass[A4paper,12pt]{amsart}
\usepackage{graphicx}
\newtheorem{thm}{Theorem}
\newtheorem*{thm*}{Theorem}
\newtheorem*{cor*}{Corollary}
\newtheorem{cor}{Corollary}
\newtheorem{lem}{Lemma}
\newtheorem{prop}{Proposition}

\newcommand{\ninf}{n\rightarrow\infty}
\newcommand{\kinf}{k\rightarrow\infty}
\newcommand{\xinf}{x\rightarrow+\infty}
\renewcommand{\Re}{\text{Re}}
\renewcommand{\Im}{\text{Im}}
\newcommand{\C}{\mathbb{C}}
\newcommand{\R}{\mathbb{R}}
\renewcommand{\l}{\left}
\renewcommand{\r}{\right}

\newcommand{\B}{\mathcal{B}}
\newcommand{\li}{\limits}
\newcommand{\as}{\text{ as }}

\begin{document}

\title[Weyl-Titchmarsh type formula for Hermite operator]{Weyl-Titchmarsh type formula for Hermite operator with small perturbation}

\author{Sergey Simonov}

\address{Department of Mathematical Physics, Institute of Physics, St. Petersburg University,
Ulianovskaia 1, 198904, St. Petergoff, St. Petersburg, Russia }
\email{sergey\_simonov@mail.ru}


\date{}

\begin{abstract}
Small perturbations of the Jacobi matrix with weights $\sqrt n$
and zero diagonal are considered. A formula relating the
asymptotics of polynomials of the first kind to the spectral
density is obtained, which is analogue of the classical
Weyl-Titchmarsh formula for the Schr\"{o}dinger operator on the
half-line with summable potential. Additionally a base of
generalized eigenvectors for "free" Hermite operator is studied
and asymptotics of Plancherel-Rotach type are obtained.
\end{abstract}

\maketitle
\section{Introduction}
Let $\{a_n\}_{n=1}^{\infty}$ be a sequence of positive numbers and
$\{b_n\}_{n=1}^{\infty}$ be a sequence of real numbers,
$\{e_n\}_{n=1}^{\infty}$ be the canonical basis in the space
$l^2(\mathbb N)$ (i.e., each vector $e_n$ has zero components
except the $n$-th which is $1$), let also $l_{fin}$ be the linear
set of sequences with finite number of non-zero components. One
can define an operator $\mathcal J$ in $l^2$, which acts in
$l_{fin}$ by the rule
\begin{equation*}
    \begin{array}{l}
      (\mathcal{J}u)_n=a_{n-1}u_{n-1}+b_nu_n+a_nu_{n+1},\ n\ge2, \\
      (\mathcal{J}u)_1=b_1u_1+a_1u_2. \\
    \end{array}
\end{equation*}
The operator is first defined on $l_{fin}$ and then the closure is
considered. Then $\mathcal J$ is self-adjoint in $l^2$ provided
$\sum\limits_{n=0}^{\infty}\frac1{a_n}=\infty$ \cite{Bz} (Carleman
condition), and it has the following matrix representation with
respect to the canonical basis:
\begin{equation*}
    \mathcal J=
    \left(%
    \begin{array}{cccc}
    b_1 & a_1 & 0 & \cdots \\

    a_1 & b_2 & a_2 & \cdots \\
    0 & a_2 & b_3 & \cdots \\
    \vdots & \vdots & \vdots & \ddots \\
    \end{array}%
    \right).
\end{equation*}
Consider the spectral equation for $\mathcal J$:
\begin{equation}\label{eq spectral general}
    a_{n-1}u_{n-1}+b_nu_n+a_nu_{n+1}=\lambda u_n,\ n\ge2.
\end{equation}
Solution $P_n(\lambda)$ of \eqref{eq spectral general} such that
$P_1(\lambda)\equiv1$, $P_2(\lambda)=\frac{\lambda-b_1}{a_1}$ is a
polynomial in $\lambda$ of degree $n-1$ and is called the
polynomial of the first kind. Correspondingly the solution
$Q_n(\lambda)$ such that $Q_1(\lambda)\equiv0$,
$Q_2(\lambda)\equiv\frac1{a_1}$ is a polynomial of degree $n-2$
and is called the polynomial of the second kind. For two solutions
of \eqref{eq spectral general} $u_n$ and $v_n$, the expression
\begin{equation*}
    W(u,v):=W(\{u_n\}_{n=1}^{\infty},\{v_n\}_{n=1}^{\infty}):=a_n(u_nv_{n+1}-u_{n+1}v_n)
\end{equation*}
is independent of $n$ and is called the (discrete) Wronskian of
$u$ and $v$. One always has
\begin{equation*}
    W(P(\lambda),Q(\lambda))\equiv1.
\end{equation*}The spectrum of
every Jacobi matrix is simple and the vector $e_1$ from the
standard basis is the generating vector \cite{Bz}. Let $dE$ be the
operator-valued spectral measure associated with $\mathcal J$.
Polynomials of the first kind are orthogonal with respect to the
measure $d\rho:=(dEe_1,e_1)$, which is also called the spectral
measure \cite{Ahiezer}. For non-real values of $\lambda$ solutions
of \eqref{eq spectral general} that belong to $l^2$ are
proportional to $Q_n(\lambda)+m(\lambda)P_n(\lambda)$
\cite{Ahiezer}, where
\begin{equation*}
    m(\lambda):=\int\limits_{\R}\frac{d\rho(x)}{x-\lambda},\
    \lambda\in\mathbb C\backslash\mathbb R
\end{equation*}
is the Weyl function. By Fatou's Theorem \cite{Koosis},
\begin{equation*}
    \rho'(\lambda)=\frac1{\pi}\Im\, m(\lambda+i0),
\end{equation*}
for a.a. $\lambda\in\mathbb R$.

In the present paper we consider small perturbations of the
operator $\mathcal J_0$, which is defined by the sequences
$\{\sqrt n\}_{n=1}^{\infty}$ and $\{0\}_{n=1}^{\infty}$:
\begin{equation*}
    \mathcal J_0=
    \left(%
    \begin{array}{cccc}
    0 & 1 & 0 & \cdots \\

    1 & 0 & \sqrt 2 & \cdots \\
    0 & \sqrt 2 & 0 & \cdots \\
    \vdots & \vdots & \vdots & \ddots \\
    \end{array}%
    \right).
\end{equation*}
Let us call $\mathcal J_0$ the "free" Hermite operator. We will
call (following \cite{Brown-Naboko-Weikard}) $\mathcal J$ the
Hermite operator, if it can be considered close to $\mathcal J_0$
in some sense. Let us call $\mathcal J$ the "small" perturbation
of $\mathcal J_0$, if $\mathcal J$ is defined by sequences
$\{a_n\}_{n=1}^{\infty}$ and $\{b_n\}_{n=1}^{\infty}$ such that
(let $c_n:=a_n-\sqrt n$)
\begin{equation}\label{eq conditions}
    c_n=o(\sqrt n)\as\ninf\text{ and }
    \sum_{n=1}^{\infty}\l(\frac{|c_n|}{n}+\frac{|c_{n+1}-c_n|+|b_n|}{\sqrt n}\r)<\infty.
\end{equation}
Denote the following expression by $\Lambda$: for any sequence
$\{u_n\}_{n=1}^{\infty}$,
\begin{equation}\label{eq definition of Lambda}
    \begin{array}{l}
        (\Lambda u)_n:=c_{n-1}u_{n-1}+b_nu_n+c_nb_{n+1},\ n\ge2, \\
        (\Lambda u)_1:=b_1u_1+c_1u_2. \\
    \end{array}
\end{equation}
Although $\Lambda$ is not a Jacobi matrix, we will write $\mathcal
J=\mathcal J_0+\Lambda$. The spectrum of $\mathcal J_0$ is purely
absolutely continuous on $\mathbb R$ with the spectral density
\begin{equation*}
    \rho_0'(\lambda)=\frac{e^{-\frac{\lambda^2}2}}{\sqrt{2\pi}}.
\end{equation*}
As it will be shown, the spectrum of $\mathcal J$ is also purely
absolutely continuous under assumption \eqref{eq conditions}.

Our goal in the present paper is to study the spectral density of
$\mathcal J$ using the asymptotic analysis of generalized
eigenvectors of $\mathcal J$ (i.e., solutions of the spectral
equation \eqref{eq spectral general}). The method is based upon
the comparison of solutions of \eqref{eq spectral general} to
solutions of the spectral equation for the free Hermite operator,
\begin{equation}\label{eq spectral equation J0}
    \sqrt{n-1}u_{n-1}+\sqrt nu_{n+1}=\lambda u_n,\ n\ge2.
\end{equation}
This is analogous to the Weyl-Titchmarsh theory for the
Schr\"{o}dinger operator on the half-line with the summable
potential. The following results will be proven (Theorem \ref{thm
result 1} in Section \ref{section free} and Theorem \ref{thm
result 2} in Section \ref{section perturbed}). Let $w$ be the
standard error function \cite{Abramowitz-Stegun}
\begin{equation}\label{eq error function}
    w(z):=\frac1{\pi i}\int_{\Gamma_z^-}\frac{e^{-\zeta^2}d\zeta}{\zeta-z}
    =-\frac1{\pi i}\int_{\Gamma_{-z}^+}\frac{e^{-\zeta^2}d\zeta}{\zeta+z},
\end{equation}
where the contours $\Gamma_z^{\pm}$ are shown on Figure \ref{fig
gamma pm z}. Function $w$ is entire.
\begin{figure}[h]
  \includegraphics{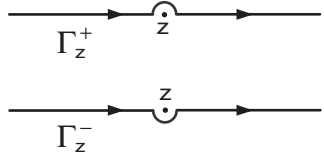}\\
  \caption{Contours $\Gamma_z^{\pm}$}
  \label{fig gamma pm z}
\end{figure}

\begin{thm}\label{thm result 1}
For every $\lambda\in\C$ equation \eqref{eq spectral equation J0}
has a basis of solutions
\begin{equation*}
    I_n^+(\lambda):=\frac{(-1)^{n-1}e^{\frac{\lambda^2}2}w^{(n-1)}\l(\frac{\lambda}{\sqrt 2}\r)}{\sqrt{(n-1)!2^{n+1}}}
\end{equation*}
and
\begin{equation*}
    I_n^-(\lambda):=\frac{e^{\frac{\lambda^2}2}w^{(n-1)}\l(-\frac{\lambda}{\sqrt 2}\r)}{\sqrt{(n-1)!2^{n+1}}},
\end{equation*}
which have the following asymptotics as $\ninf$:
    \begin{equation*}
        I^{\pm}_n(\lambda)
        =\frac{(\mp i)^{n-1} e^{\frac{\lambda^2}4\pm i\lambda\sqrt{n}}}{(8\pi n)^{1/4}}
        \l(1+O\l(\frac1{\sqrt n}\r)\r).
    \end{equation*}
These asymptotics are uniform with respect to $\lambda$ in every
bounded set in $\mathbb C$. Polynomials of the first kind for
$\mathcal J_0$ are related to $I_n^{\pm}$ in the following way:
\begin{equation*}
    {P_0}_n(\lambda)=I_n^+(\lambda)+I_n^-(\lambda).
\end{equation*}
\end{thm}

\begin{thm}\label{thm result 2}
    Let the conditions \eqref{eq conditions} hold for $\mathcal J$.
    Then
    \\
    1. For every $\lambda\in\overline{\mathbb C_+}$ there exists
    \begin{equation*}
        F(\lambda):=1+i\sqrt{2\pi}e^{-\frac{\lambda^2}2}\sum_{n=1}^{\infty}(\Lambda I^+(\lambda))_nP_n(\lambda)
    \end{equation*}
    (the Jost function), which is analytic function in $\C_+$ and continuous in $\overline{\C_+}$.
    \\
    2. Polynomials of the first kind have the following asymptotics as $\ninf$:
    \begin{itemize}
        \item
            For $\lambda\in\C_+$,
            \begin{equation*}
                    P_n(\lambda)=F(\lambda)I_n^-(\lambda)+o\l(\frac{e^{\Im \lambda\sqrt n}}{n^{1/4}}\r)\as\ninf,
            \end{equation*}
        \item
            For $\lambda\in\R$,
            \begin{equation*}
                P_n(\lambda)=F(\lambda)I_n^-(\lambda)+\overline{F(\lambda)}I_n^+(\lambda)+o(n^{-\frac14})\as\ninf.
            \end{equation*}
    \end{itemize}
3. The spectrum of $\mathcal J$ is purely absolutely continuous, and for a.a. $\lambda\in\mathbb R$
    \begin{equation*}
        \rho'(\lambda)=\frac{e^{-\frac{\lambda^2}2}}{\sqrt{2\pi}|F(\lambda)|^2}
    \end{equation*}
    (the Weyl-Titchmarsh type formula).
\end{thm}

The idea of the Weyl-Titchmarsh type formula is the relation
between the spectral density and the behavior of $P_n(\lambda)$
for large values of $n$. We can formulate this in the form of the
corollary.

\begin{cor}\label{cor result}
Let the conditions \eqref{eq conditions} hold for $\mathcal J$.
Then the spectrum of $\mathcal J$ is purely absolutely continuous
and the spectral density equals for a.a. $\lambda\in\mathbb R$
\begin{equation*}
    \rho'(\lambda)=\frac1{\pi}\lim_{\ninf}\frac1{\sqrt n(P_n^2(\lambda)+P_{n+1}^2(\lambda))},
\end{equation*}
the right-hand side being finite and non-zero for every
$\lambda\in\R$.
\end{cor}

Theorem \ref{thm result 2} can be proven by another method, based
on the Levinson-type analytical and smooth theorem, cf.
\cite{Coddington-Levinson} and papers of Bernzaid-Lutz
\cite{Bernzaid-Lutz}, Janas-Moszy\'{n}ski \cite{JM1} and Silva
\cite{Silva uniform}, \cite{Silva smooth}. None of their results
is directly applicable here, and the approach of the present paper
is different.

The considered situation is parallel to the Weyl-Titchmarsh theory
for Schr\"{o}dinger operator on the half-line with summable
potential. Let $q$ be a real-valued function on $\mathbb R_+$ and
$q\in L_1(\mathbb R_+)$. Consider the Schr\"{o}dinger operator on
$\mathbb R_+$
\begin{equation*}
    \mathcal L=-\frac{d^2}{dx^2}+q(x)
\end{equation*}
with the Dirichlet boundary condition. The purely absolutely
continuous spectrum of $\mathcal L$ coincides with $\mathbb{R}_+$
\cite{Titchmarsh}. Let $\varphi(x,\lambda)$ be a solution of the
spectral equation for $\mathcal L$,
\begin{equation*}
    -u''(x,\lambda)+q(x)u(x,\lambda)=\lambda u(x,\lambda),
\end{equation*}
such that $\varphi(0,\lambda)\equiv0$,
$\varphi'(0,\lambda)\equiv1$ (satisfying the boundary condition).
The following result holds \cite{Titchmarsh}.

\begin{prop}
    If $q\in L_1(\mathbb R_+)$, then for every $k>0$ there exist
    $a(k)$ and $b(k)$ such that
    \begin{equation*}
        \varphi(x,k^2)=a(k)\cos(kx)+b(k)\sin(kx)+o(1)\as\xinf,
    \end{equation*}
    and for a.a. $\lambda>0$
    \begin{equation*}
        \rho'(\lambda)=\frac1{\pi \sqrt{\lambda} (a^2(\sqrt{\lambda})+b^2(\sqrt{\lambda}))}
    \end{equation*}
    (the classical Weyl-Titchmarsh formula).
\end{prop}

Solutions $I_n^+(\lambda)$ and $I_n^-(\lambda)$ are the direct
analogues to the solutions $\frac{e^{ikx}}{2ik}$ and
$\frac{e^{-ikx}}{-2ik}$ of the spectral equation for "free"
Schr\"{o}dinger operator,
\begin{equation*}
    -u''(x,k^2)=k^2u(x,k^2).
\end{equation*}
The main technical difficulty of our problem is non-triviality of
solutions $I_n^{\pm}(\lambda)$ compared to $\frac{e^{\pm
ikx}}{\pm 2ik}$. The model of the Hermite operator was studied in
the paper of Brown-Naboko-Weikard \cite{Brown-Naboko-Weikard}, but
solutions $I_n^{\pm}(\lambda)$ were not introduced there.

\section{The free Hermite operator}\label{section free}
In this section we study asymptotic properties of generalized
eigenvectors for $J_0$ and prove Theorem \ref{thm result 1}. Let
us give its formulation again.

\begin{thm*}
For every $\lambda\in\C$ equation \eqref{eq spectral equation J0}
has a basis of solutions
\begin{equation}\label{eq In^+}
    I_n^+(\lambda):=\frac{(-1)^{n-1}e^{\frac{\lambda^2}2}w^{(n-1)}\l(\frac{\lambda}{\sqrt 2}\r)}{\sqrt{(n-1)!2^{n+1}}}
\end{equation}
and
\begin{equation}\label{eq In^-}
    I_n^-(\lambda):=\frac{e^{\frac{\lambda^2}2}w^{(n-1)}\l(-\frac{\lambda}{\sqrt 2}\r)}{\sqrt{(n-1)!2^{n+1}}},
\end{equation}
which have the following asymptotics as $\ninf$:
    \begin{equation}\label{eq asymptotics of In^pm}
        I^{\pm}_n(\lambda)
        =\frac{(\mp i)^{n-1} e^{\frac{\lambda^2}4\pm i\lambda\sqrt{n}}}{(8\pi n)^{1/4}}
        \l(1+O\l(\frac1{\sqrt n}\r)\r).
    \end{equation}
These asymptotics are uniform with respect to $\lambda$ in every
bounded set in $\mathbb C$. Polynomials of the first kind for
$\mathcal J_0$ are related to $I_n^{\pm}$ in the following way:
\begin{equation}\label{eq relation between P and In^pm}
    {P_0}_n(\lambda)=I_n^+(\lambda)+I_n^-(\lambda).
\end{equation}
\end{thm*}

\begin{proof}
The spectral equation \eqref{eq spectral equation J0} for
$\mathcal J_0$,
\begin{equation*}
    \sqrt{n-1}u_{n-1}+\sqrt{n}u_{n+1}=\lambda u_n,\ n \geq 2,
\end{equation*}
can be transformed to the recurrence relation
\begin{equation}\label{eq recurrent relation for Hermite polynomials}
    2nv_{n-1}(x)+v_{n+1}(x)=2xv_n(x),\ n\ge1
\end{equation}
if one takes $v_n:=\sqrt{2^nn!}u_{n+1}$ and
$x:=\frac{\lambda}{\sqrt 2}$. Equation \eqref{eq recurrent
relation for Hermite polynomials} is satisfied by Hermite
polynomials \cite{Abramowitz-Stegun}, and this means (together
with initial values: $H_0(x)\equiv1$ and $H_1(x)=2x$), that the
polynomials of the first kind for $J_0$ equal
\begin{equation}\label{eq connection}
    {P_0}_n(\lambda)=\frac{H_{n-1}(\frac{\lambda}{\sqrt2})}{\sqrt{2^{n-1}(n-1)!}}.
\end{equation}
Equation \eqref{eq recurrent relation for Hermite polynomials} has
two other linearly independent solutions, $w^{(n)}(-x)$ and
$(-1)^nw^{(n)}(x)$ \cite{Abramowitz-Stegun}. This can be checked
by substituting them into \eqref{eq recurrent relation for Hermite
polynomials} using the formula
\begin{equation*}
    w^{(n)}(z)=\frac{n!}{\pi i}\int_{\Gamma_z^-}\frac{e^{-\zeta^2}d\zeta}{(\zeta-z)^{n+1}}
\end{equation*}
and integrating by parts. From the integral representation for
Hermite polynomials \cite{Abramowitz-Stegun},
\begin{multline*}
    H_n(x)=\frac{n!}{2\pi i} \oint_0\frac{e^{2xz-z^2}}{z^{n+1}}dz
    \\
    =\frac{n!e^{x^2}}{2\pi i}\l(\int_{\Gamma_{-x}^-}\frac{e^{-z^2}dz}{(z+x)^{n+1}}-\int_{\Gamma_{-x}^+}\frac{e^{-z^2}dz}{(z+x)^{n+1}}\r)
    \\
    =\frac{e^{x^2}}2w^{(n)}(-x)+\frac{e^{x^2}}2(-1)^nw^{(n)}(x),
\end{multline*}
where the contour $\Gamma_z^+$ is shown on Figure \ref{fig gamma
pm z}. Correspondingly, equation \eqref{eq spectral equation J0}
has two linearly independent solutions of the form \eqref{eq In^+}
and \eqref{eq In^-} and relation \eqref{eq relation between P and
In^pm} holds. Asymptotics of these solutions follow immediately
from Corollary \ref{cor asymptotics of w^(n-1)(z)} from the
appendix.
\end{proof}

In what follows we will need to know the Wronskian of the
solutions.

\begin{lem}
    \begin{equation*}
        W(I^{+}(\lambda),I^{-}(\lambda))=i\frac{e^{\frac{\lambda^2}2}}{\sqrt{2\pi}}.
    \end{equation*}
\end{lem}

\begin{proof}
One has from \eqref{eq In^+} and \eqref{eq In^-}:
\begin{multline*}
    W(I^{+}(\lambda),I^{-}(\lambda))=I_1^+(\lambda)I_2^-(\lambda)-I_2^+(\lambda)I_1^-(\lambda)
    \\
    =\frac{e^{\lambda^2}}{4\sqrt 2}
    \l(w\l(\frac{\lambda}{\sqrt 2}\r)w'\l(-\frac{\lambda}{\sqrt 2}\r)+w\l(-\frac{\lambda}{\sqrt 2}\r)w'\l(\frac{\lambda}{\sqrt 2}\r)\r)
    =i\frac{e^{\frac{\lambda^2}2}}{\sqrt{2\pi}},
\end{multline*}
using the following properties of the error function \cite{Abramowitz-Stegun}:
\begin{equation*}
    \begin{array}{l}
      w'(z)=-2zw(z)+\frac{2i}{\sqrt{\pi}}, \\
      w(z)+w(-z)=2e^{-z^2}. \\
    \end{array}
\end{equation*}
\end{proof}

\section{The perturbed Hermite operator}\label{section perturbed}
In this section, we consider the Hermite operator $\mathcal J$
with "small" perturbation, i.e., satisfying conditions \eqref{eq
conditions}, and prove Theorem \ref{thm result 2}. We study
asymptotics of polynomials of the first and second kind using the
Volterra-type equation and derive from these asymptotics a formula
for the Weyl function. The desired Weyl-Titchmarsh type formula
follows from this. We start with proving a formula of variation of
parameters. Remind that $P_n(\lambda)$ are polynomials of the
first kind for $\mathcal J$, ${P_0}_n(\lambda)$ are polynomials of
the first kind for $\mathcal J_0$, $\Lambda$ is the expression
given by \eqref{eq definition of Lambda}, $a_n=\sqrt n+c_n$. Let
us denote
\begin{equation*}
    W(\lambda):=W(I^+(\lambda),I^-(\lambda)).
\end{equation*}

\begin{lem}\label{lem variation of parameters formula}
    For $n\ge2$,
    \begin{equation}\label{eq formula of variation of parameters}
        \frac{a_{n-1}}{\sqrt{n-1}}P_n(\lambda)
        ={P_0}_n(\lambda)-\sum_{k=1}^{n-1}\frac{(\Lambda I^+(\lambda))_kI^-_n(\lambda)-I^+_n(\lambda)(\Lambda I^-(\lambda))_k}{W(\lambda)}P_k(\lambda).
    \end{equation}
\end{lem}

\begin{proof}
Let us omit the dependence on $\lambda$ everywhere. First let us prove that
\begin{equation}\label{eq formula of variation, first step}
    P_n=u_n-\sum_{k=2}^{n-1}\frac{I_k^{+}I_n^{-}-I_k^{-}I_n^{+}}W(\Lambda P)_k,\ n\ge3,
\end{equation}
where $u$ is the solution of \eqref{eq spectral equation J0} such
that $u_1=P_1$ and $u_2=P_2$. Let us denote
\begin{equation*}
    \widetilde P_n:=
    \l\{
    \begin{array}{l}
      u_n-\sum\limits_{k=2}^{n-1}\frac{I_k^{+}I_n^{-}-I_k^{-}I_n^{+}}W(\Lambda P)_k,\ n\ge3, \\
      P_n,\ n=1,2 \\
    \end{array}
    \r.
\end{equation*}
I fact, one has to check that
\begin{equation*}
    \sqrt{n-1}\widetilde P_{n-1}-\lambda\widetilde P_n+\sqrt n\widetilde P_{n+1}=-(\Lambda P)_n,\ n\ge2,
\end{equation*}
(this non-homogeneous equation has only one solution with fixed
two first values, so $\widetilde P$ should coincide with $P$).
Since $u$, $I^+$ and $I^-$ are solutions to \eqref{eq spectral
equation J0} and
\begin{equation*}
    \sqrt n \sum_{k=n}^n \frac{I^+_kI^-_{n+1}-I^+_{n+1}I^-_k}W(\Lambda P)_k=(\Lambda P)_n,
\end{equation*}
the previous is equivalent to
\begin{equation*}
    -\lambda\sum_{k=n-1}^{n-1}\frac{I^+_kI^-_n-I^+_nI^-_k}W(\Lambda P)_k+\sqrt n\sum_{k=n-1}^{n-1}\frac{I^+_kI^-_{n+1}-I^+_{n+1}I^-_k}W(\Lambda P)_k=0
\end{equation*}
for $n\ge3$. The latter is true, because $-\lambda I^{\pm}_n+\sqrt
nI^{\pm}_{n+1}=-\sqrt{n-1}I^{\pm}_{n-1}$.

After shifting indices in different parts of the sum in \eqref{eq
formula of variation, first step} one obtains:
\begin{multline*}
    P_n=u_n+\frac{I^+_1I^-_n-I^+_nI^-_1}W(b_1P_1+c_1P_2)
    \\
    -\sum_{k=1}^{n-1}\frac{(\Lambda I^+)_kI^-_n-I^+_n(\Lambda I^-)_k}WP_k-\frac{c_{n-1}}{\sqrt{n-1}}P_n.
\end{multline*}
Since
\begin{equation*}
    u_n=\frac{I^+_1P_2-I^+_2P_1}WI^-_n-\frac{I^-_1P_2-I^-_2P_1}WI^+_n,\ P_1=1,\ P_2=\frac{\lambda-b_1}{a_1},
\end{equation*}
one has:
\begin{multline*}
    u_n+\frac{I^+_1I^-_n-I^+_nI^-_1}W(b_1P_1+c_1P_2)=\frac{\lambda I^+_1-I^+_2}WI^-_n-\frac{\lambda I^-_1-I^-_2}WI^+_n
    \\
    =\frac{I^+_1{P_0}_2-I^+_2{P_0}_1}WI^-_n-\frac{I^+_2{P_0}_1-I^+_1{P_0}_2}WI^+_n={P_0}_n.
\end{multline*}
Therefore
\begin{equation*}
\frac{a_{n-1}}{\sqrt{n-1}}P_n={P_0}_n-\sum_{k=1}^{n-1}\frac{(\Lambda I^+)_kI^-_n-I^+_n(\Lambda I^-)_k}{W}P_k.
\end{equation*}
\end{proof}

Equation \eqref{eq formula of variation of parameters} is of Volterra type. We need the following standard lemma to deal with it. Consider the Banach space
\begin{equation*}
    \mathcal{B}:=\l\{\{u_n\}_{n=1}^{\infty}:\sup\limits_n\l(\frac{|u_n|n^{1/4}}{e^{|\Im\,\lambda|\sqrt n}}\r)<\infty\r\}
\end{equation*}
with the norm
\begin{equation*}
    \|u\|_{\mathcal{B}}:=\sup\limits_n\l(\frac{|u_n|n^{1/4}}{e^{|\Im\,\lambda|\sqrt n}}\r)
\end{equation*}
(we omit the dependence on $\lambda$ in the notation for $\mathcal
B$). Let $\mathcal V$ be the expression
\begin{equation}\label{def operator V}
    (\mathcal Vu)_n:=\l\{
    \begin{array}{l}
        0,\ n=1 \\
        \sum_{k=1}^{n-1}V_{nk}u_k,\ n\ge2 \\
    \end{array}
    \r.
\end{equation}
for any sequence $\{u\}_{n=1}^{\infty}$. Let
\begin{equation}\label{eq nu}
    \nu:=\sup\limits_{n>1}\sum\limits_{k=1}^{n-1}|V_{nk}|e^{|\Im\, \lambda|(\sqrt k-\sqrt n)}\l(\frac nk\r)^{1/4}.
\end{equation}

\begin{lem}\label{lem bounds for abstract Volterra operator}
    If $\nu<\infty$, then $\mathcal V$ is a bounded operator in $\mathcal B$,
    $(I-\mathcal V)^{-1}$ exists and $\|\mathcal V\|_{\mathcal{B}}\le\nu$, $\|(I-\mathcal V)^{-1}\|_{\mathcal{B}}\le e^{\nu}$.
\end{lem}

\begin{proof}
By definition of the operator norm we have to check the finiteness
of the following:
\begin{multline*}
    \sup\li_{u\neq0}\frac{\|\mathcal Vu\|_{\B}}{\|u\|_{\B}}
    =\sup\li_{u\neq0}
    \frac{\sup\li_{n>1}\frac{\l|\sum_{k=1}^{n-1}V_{nk}u_k\r|n^{1/4}}{e^{|\Im\, \lambda|\sqrt n}}}
         {\sup\li_{n}\frac{|u_n|n^{1/4}}{e^{|\Im\, \lambda|\sqrt n}}}
    \\
    \le\sup\li_{u\neq0}
    \frac{\sup\li_{n>1}\sum\li_{k=1}^{n-1}|V_{nk}|\frac{|u_k|k^{1/4}}{e^{|\Im\, \lambda|\sqrt k}}\l(\frac nk\r)^{1/4}e^{|\Im\, \lambda|(\sqrt k-\sqrt n)}}
         {\sup\li_{n}\frac{|u_n|n^{1/4}}{e^{|\Im\, \lambda|\sqrt n}}}.
\end{multline*}
Denoting $\widetilde{u}_n:=u_n\frac{n^{1/4}}{e^{\Im\lambda\sqrt
n}}$, we have:
\begin{multline*}
    \sup\li_{u\neq0}\frac{\|\mathcal Vu\|_{\B}}{\|u\|_{\B}}
    \le\sup\li_{\widetilde{u}\neq0}
    \frac{\sup\li_{n>1}\sum\li_{k=1}^{n-1}|V_{nk}||\widetilde{u}_k|\l(\frac nk\r)^{1/4}e^{|\Im\, \lambda|(\sqrt k-\sqrt n)}}
         {\sup\li_{n}|\widetilde{u}_n|}
    \\
    \le\sup\li_{n>1}\sum_{k=1}^{n-1}|V_{nk}|\l(\frac nk\r)^{1/4}e^{|\Im\, \lambda|(\sqrt k-\sqrt
    n)},
\end{multline*}
hence $\mathcal V$ is bounded. Quite similarly,
\begin{multline*}
    \|\mathcal V^l\|_{\B}
    \le\sup\li_{n>1}\sum_{k=1}^{n-1}\l|\sum_{1\le k_1<k_2<...<k_{l-1}<k}V_{nk_1}V_{k_1k_2}...V_{k_{l-1}k}\r|\l(\frac nk\r)^{1/4}e^{|\Im\, \lambda|(\sqrt k-\sqrt n)}
    \\
    \le\sup\li_{n>1}\frac{\l(\sum_{k=1}^{n-1}|V_{nk}|\l(\frac nk\r)^{1/4}e^{|\Im\, \lambda|(\sqrt k-\sqrt n)}\r)^l}{l!}.
\end{multline*}
Therefore
\begin{equation*}
    1+\|\mathcal V\|_{\B}+\|\mathcal V^2\|_{\B}+...
    \le\exp\l\{\sup\limits_{n>1}\sum\limits_{k=1}^{n-1}|V_{nk}|e^{|\Im\, \lambda|(\sqrt k-\sqrt n)}\l(\frac
    nk\r)^{1/4}\r\},
\end{equation*}
and hence the operator $(I-\mathcal V)^{-1}$ exists, is bounded,
and its norm is estimated by the same expression.
\end{proof}

Now we can prove the uniform estimate on the growth of polynomials.

\begin{lem}\label{lem estimates on growth of polunomials}
    Let the condition \eqref{eq conditions} hold for $\mathcal J$. Then
    \begin{equation}\label{eq estimate for P}
        P_n(\lambda)=O\l(\frac{e^{|\Im\, \lambda|\sqrt n}}{n^{1/4}}\r)\as\ninf
    \end{equation}
    uniformly with respect to $\lambda$ on every bounded set in $\C$.
\end{lem}

\begin{proof}
Let us rewrite \eqref{eq formula of variation of parameters} as
\begin{equation*}
    P(\lambda)=v(\lambda)+\mathcal V(\lambda)P(\lambda),
\end{equation*}
where
\begin{equation*}
    \begin{array}{l}
    v_n(\lambda):=
    \l\{
        \begin{array}{l}
        1,\ n=1,\\
        \frac{\sqrt{n-1}}{a_{n-1}}{P_0}_n(\lambda),\ n\ge2,\\
        \end{array}
    \r.\\
    (\mathcal V(\lambda))_n:=
    \l\{
        \begin{array}{l}
        0,\ n=1,\\
        -\frac{\sqrt{n-1}}{a_{n-1}}\sum_{k=1}^{n-1}\frac{(\Lambda I^+(\lambda))_kI^-_n(\lambda)-
        I^+_n(\lambda)(\Lambda I^-(\lambda))_k}{W(\lambda)}u_k,\ n\ge2.\\\
        \end{array}
    \r.\\
    \end{array}
\end{equation*}
What we need to prove is that $P(\lambda)\in\mathcal B$ and
$\|P(\lambda)\|_{\mathcal{B}}$ is bounded on every bounded set in
$\C$. It will suffice to prove the same for
$\|v(\lambda)\|_{\mathcal{B}}$ and for $\nu(\lambda)$ related to
$\mathcal V(\lambda)$ by \eqref{eq nu}, due to Lemma \ref{lem
bounds for abstract Volterra operator}. First follows from the
asymptotics given by Theorem \ref{thm result 1}, so consider the
second. The kernel of $\mathcal V(\lambda)$ is
\begin{equation*}
    V_{nk}(\lambda):=-\frac{\sqrt{n-1}}{a_{n-1}}\frac{(\Lambda I^+(\lambda))_kI^-_n(\lambda)-I^+_n(\lambda)(\Lambda I^-(\lambda))_k}{W(\lambda)},\ 1\le k\le n-1.
\end{equation*}
Fix a bounded set $K\subset\C$. It follows from \eqref{eq asymptotics of In^pm} that
\begin{multline}\label{eq estimate for Lamba I^pm}
    (\Lambda I^{\pm}(\lambda))_k
    =\frac{\l|e^{\frac{\lambda^2}4}\r|}{(8\pi)^{\frac14}}\l|c_{k-1}\frac{i^{k-1}e^{i\lambda\sqrt{k-1}}}{(k-1)^{\frac14}}
    +b_k\frac{i^ke^{i\lambda\sqrt k}}{k^{\frac14}}
    +c_k\frac{i^{k+1}e^{i\lambda\sqrt{k+1}}}{(k+1)^{\frac14}}\r|
    \\
    +O\l(\frac{|c_{k-1}|+|b_k|+|c_k|}{k^{\frac34}}e^{\mp\Im\lambda\sqrt k}\r)
    \\
    =O\l(\l(|b_k|+|c_k-c_{k-1}|+\frac{|c_k|}{\sqrt k}\r)\frac{e^{\mp\Im\lambda\sqrt
    k}}{k^{\frac14}}\r)\as\kinf
\end{multline}
uniformly with respect to $\lambda\in K$. Hence there exists $C_1$
such that
\begin{multline*}
    |(\Lambda I^+(\lambda))_kI^-_n(\lambda)|, |I^+_n(\lambda)(\Lambda I^-(\lambda))_k|
    \\
    <C_1(\l(|b_k|+|c_k-c_{k-1}|+\frac{|c_k|}{\sqrt k}\r)\frac{e^{|\Im\, \lambda||\sqrt n-\sqrt
    k|}}{(nk)^{1/4}}
\end{multline*}
for every $n,k\in\mathbb{N}$. Therefore there exists $C_2$ such
that
\begin{equation*}
    \nu(\lambda)<C_2\sum_{k=1}^{\infty}\l(\frac{|c_k|}{k}+\frac{|c_k-c_{k-1}|+|b_k|}{\sqrt k}\r),
\end{equation*}
and this estimate is uniform with respect to $\lambda\in K$. This completes the proof.
\end{proof}

It is possible now to introduce the Jost function and to find the asymptotics of the polynomials.

\begin{lem}\label{lem asymptotics of polynomials}
    Let the condition \eqref{eq conditions} hold for $\mathcal J$. Then the function
    \begin{equation}\label{eq Jost function}
        F(\lambda):=1+i\sqrt{2\pi}e^{-\frac{\lambda^2}2}\sum_{n=1}^{\infty}(\Lambda I^+(\lambda))_nP_n(\lambda)
    \end{equation}
    is analytic in $\C_+$ and continuous in $\overline{\C_+}$.
    Polynomials of the first kind for $\mathcal J$, $P_n(\lambda)$, have the following asymptotics as $\ninf$:
    \begin{itemize}
        \item
            For $\lambda\in\C_+$,
            \begin{equation}\label{eq asymptotics C+}
                    P_n(\lambda)=F(\lambda)I_n^-(\lambda)+o\l(\frac{e^{\Im \lambda\sqrt n}}{n^{1/4}}\r)\as\ninf,
            \end{equation}
        \item
            For $\lambda\in\R$,
            \begin{equation}\label{eq asymptotics R}
                P_n(\lambda)=F(\lambda)I_n^-(\lambda)+\overline{F(\lambda)}I_n^+(\lambda)+o(n^{-\frac14})\as\ninf.
            \end{equation}
    \end{itemize}
\end{lem}

\begin{proof}
Let us rewrite \eqref{eq formula of variation of parameters} as
\begin{multline}\label{eq intermediate}
    P_n(\lambda)\frac{a_{n-1}}{\sqrt{n-1}}
    \\
    =\l(1+\sum_{k=1}^{n-1}\frac{(\Lambda I^-(\lambda))_kP_k(\lambda)}{W(\lambda)}\r)I_n^+(\lambda)
    +\l(1-\sum_{k=1}^{n-1}\frac{(\Lambda I^+(\lambda))_kP_k(\lambda)}{W(\lambda)}\r)I_n^-(\lambda).
\end{multline}
From the estimates on $(\Lambda I^+(\lambda))_k$ and
$P_k(\lambda)$ \eqref{eq estimate for Lamba I^pm} and \eqref{eq
estimate for P} it follows that
\begin{equation*}
    (\Lambda I^+(\lambda))_kP_k(\lambda)=O\l(\frac{|c_k|}k+\frac{|c_{k+1}-c_k|+|b_k|}{\sqrt k}\r)\as\kinf
\end{equation*}
uniformly with respect to $\lambda$ on every compact subset of
$\overline{\C_+}$. Hence the expression
\begin{equation*}
    F_n(\lambda):=1-\sum_{k=1}^{n-1}\frac{(\Lambda I^+(\lambda)_kP_k(\lambda)}{W(\lambda)}
\end{equation*}
converges as $\ninf$ to the function
\begin{equation*}
    F(\lambda):=1-\sum_{k=1}^{\infty}\frac{(\Lambda I^+(\lambda))_kP_k(\lambda)}{W(\lambda)}
\end{equation*}
analytic in $\C_+$ and continuous in $\overline{\C_+}$.

Consider $\lambda\in\C_+$. The first term in \eqref{eq
intermediate} is relatively small. Indeed,
\begin{multline*}
    \l|\frac{I_n^+(\lambda)\l(1+\sum\limits_{k=1}^{n-1}\frac{(\Lambda I^-(\lambda)_kP_k(\lambda)}{W(\lambda)}\r)}{I_n^-(\lambda)}\r|=
    \\
    =O\l(e^{-2\Im\lambda\sqrt n}+\sum_{k=1}^{n-1}e^{2\Im\lambda(\sqrt k-\sqrt n)}\l(\frac{|c_k|}k+\frac{|c_{k+1}-c_k|+|b_k|}{\sqrt k}\r)\r)=o(1)
\end{multline*}
as $\ninf$. This means that
$\frac{P_n(\lambda)}{I_n^-(\lambda)}\rightarrow F(\lambda)$ as
$\ninf$.

Consider $\lambda\in\R$. Equation \eqref{eq intermediate} yields:
\begin{multline*}
    P_n(\lambda)
    =\l(F_n(\lambda)\frac{\sqrt{n-1}}{a_{n-1}}\r)I_n^-(\lambda)
    +\l(\overline{F_n(\lambda)}\frac{\sqrt{n-1}}{a_{n-1}}\r)I_n^+(\lambda).
    \\
    =F(\lambda)I_n^-(\lambda)+\overline{F(\lambda)}I_n^+(\lambda)+o(n^{-\frac14})\as\ninf
\end{multline*}
due to asymptotics \eqref{eq asymptotics of In^pm} of
$I_n^{\pm}(\lambda)$ and the convergence of $F_n(\lambda)$. The
proof is complete.
\end{proof}

The final step is the proof of the absolute continuity of the
spectrum of $\mathcal J$ and the formula for the spectral density.

\begin{lem}\label{lem formula for the spectral density}
    Let the condition \eqref{eq conditions} hold for $\mathcal J$. Then the
    spectrum of $\mathcal J$ is purely absolutely continuous
    and for a.a. $\lambda\in\mathbb{R}$ the following formula
    holds:
    \begin{equation}\label{eq formula for the spectral density}
        \rho'(\lambda)=\frac{e^{-\frac{\lambda^2}2}}{\sqrt{2\pi}|F(\lambda)|^2},
    \end{equation}
    where $F(\lambda)$ defined by \eqref{eq Jost function} and does not vanish on $\R$.
\end{lem}

\begin{proof}
Polynomials of the second kind have asymptotics of the same type
as polynomials of the first kind. Cropped Jacobi matrix $\mathcal
J_1$ being the original one $\mathcal J$ with the first row and
the first column removed,
\begin{equation*}
    \mathcal J_1=
    \left(%
    \begin{array}{cccc}
    b_2 & a_2 & 0 & \cdots \\
    a_2 & b_3& a_3 & \cdots \\
    0 & a_3 & b_4 & \cdots \\
    \vdots & \vdots & \vdots & \ddots \\
    \end{array}%
    \right),
\end{equation*}
satisfies conditions of Lemma \ref{lem asymptotics of
polynomials}. And the polynomials $a_1Q_n(\lambda)$ are the
polynomials of the first kind for $\mathcal J_1$, so there exists
a function $F_1(\lambda)$, analytic in $\C_+$ and continuous in
$\overline{\C_+}$, such that
\begin{itemize}
    \item
        For $\lambda\in\C_+$, $Q_n(\lambda)=F_1(\lambda)I_n^-(\lambda)+o\l(\frac{e^{\Im
        \lambda\sqrt n}}{n^{1/4}}\r)$ as $\ninf$,
    \item
        For $\lambda\in\R$, $Q_n(\lambda)=F_1(\lambda)I_n^-(\lambda)+\overline{F_1(\lambda)}I_n^+(\lambda)+o(n^{-\frac14})$ as
        $\ninf$.
\end{itemize}
The combination $Q_n(\lambda)+m(\lambda)P_n(\lambda)$ belongs to
$l^2$ for $\lambda\in\C_+$, hence
\begin{equation*}
    m(\lambda)=-\frac{F_1(\lambda)}{F(\lambda)}\text{ for }\lambda\in\C_+.
\end{equation*}
Consider $\lambda\in\R$. One has:
\begin{multline*}
    1=W(P,Q)=(\sqrt n+c_n)(P_nQ_{n+1}-P_{n+1}Q_n)
    \\
    =\sqrt n(I^+_nI^-_{n+1}-I^+_{n+1}I^-_n)(\overline FF_1-F\overline F_1)+o(1)
    \\
    =W(I^+,I^-)(\overline FF_1-F\overline F_1),
\end{multline*}
therefore
\begin{equation*}
    F_1(\lambda)\overline{F(\lambda)}-\overline{F_1(\lambda)}F(\lambda)=-i\sqrt{2\pi}e^{-\frac{\lambda^2}2}
\end{equation*}
for $\lambda\in\mathbb R$, and hence for every
$\lambda\in\overline{\C_+}$. It follows that $F(\lambda)$ and
$F_1(\lambda)$ do not have zeros in $\overline{\C_+}$. For every
$\lambda\in\mathbb R$ there exists the finite limit
\begin{equation*}
    m(\lambda+i0)=-\frac{F_1(\lambda)}{F(\lambda)},
\end{equation*}
which is continuous in $\lambda$. It follows then \cite{KP} that
the spectrum of $\mathcal J$ is purely absolutely continuous and
the spectral density equals for a.a. $\lambda\in\mathbb R$
\begin{equation*}
    \rho'(\lambda)=\frac1{\pi}\Im\, m(\lambda+i0)
    =\frac{F(\lambda)\overline{F_1(\lambda)}-\overline{F(\lambda)}F_1(\lambda)}{2\pi
    i|F(\lambda)|^2}=\frac{e^{-\frac{\lambda^2}2}}{\sqrt{2\pi}|F(\lambda)|^2},
\end{equation*}
which completes the proof.
\end{proof}

Theorem \ref{thm result 2} follows directly from Lemmas \ref{lem
asymptotics of polynomials} and \ref{lem formula for the spectral
density}. Let us repeat its formulation.

\begin{thm*}
    Let the conditions \eqref{eq conditions} hold for $\mathcal J$.
    Then
    \\
    1. For every $\lambda\in\overline{\mathbb C_+}$ there exists
    \begin{equation*}
        F(\lambda):=1+i\sqrt{2\pi}e^{-\frac{\lambda^2}2}\sum_{n=1}^{\infty}(\Lambda I^+(\lambda))_nP_n(\lambda)
    \end{equation*}
    (the Jost function), which is analytic function in $\C_+$ and continuous in $\overline{\C_+}$.
    \\
    2. Polynomials of the first kind have the following asymptotics as $\ninf$:
    \begin{itemize}
        \item
            For $\lambda\in\C_+$,
            \begin{equation*}
                    P_n(\lambda)=F(\lambda)I_n^-(\lambda)+o\l(\frac{e^{\Im \lambda\sqrt n}}{n^{1/4}}\r)\as\ninf,
            \end{equation*}
        \item
            For $\lambda\in\R$,
            \begin{equation*}
                P_n(\lambda)=F(\lambda)I_n^-(\lambda)+\overline{F(\lambda)}I_n^+(\lambda)+o(n^{-\frac14})\as\ninf.
            \end{equation*}
    \end{itemize}
3. The spectrum of $\mathcal J$ is purely absolutely continuous,
and for a.a. $\lambda\in\mathbb R$
    \begin{equation*}
        \rho'(\lambda)=\frac{e^{-\frac{\lambda^2}2}}{\sqrt{2\pi}|F(\lambda)|^2}.
    \end{equation*}
\end{thm*}

It remains to prove the following corollary.

\begin{cor*}
    Let the conditions \eqref{eq conditions} hold for $\mathcal J$.
    Then the spectrum of $\mathcal J$ is purely absolutely continuous
    and the spectral density equals for a.a. $\lambda\in\mathbb R$
    \begin{equation}\label{eq rho in terms of polynomials}
        \rho'(\lambda)=\frac1{\pi}\lim_{\ninf}\frac1{\sqrt n(P_n^2(\lambda)+P_{n+1}^2(\lambda))},
    \end{equation}
    the right-hand side being finite and non-zero for every
    $\lambda\in\R$.
\end{cor*}

\begin{proof}
From the asymptotics \eqref{eq asymptotics R} and \eqref{eq
asymptotics of In^pm} one has for $\lambda\in\mathbb R$:
\begin{equation*}
    P_n^2(\lambda)+P_{n+1}^2(\lambda)=\frac{4|F(\lambda)|^2e^{\frac{\lambda^2}2}}{\sqrt{8\pi n}}
    +o\l(\frac1{\sqrt n}\r),
\end{equation*}
so
\begin{equation*}
    \frac1{|F(\lambda)|^2}=\sqrt{\frac2{\pi}}\frac{e^{\frac{\lambda^2}2}}{\lim\limits_{\ninf}\sqrt
    n(P_n^2(\lambda)+P_{n+1}^2)}.
\end{equation*}
Substituting into \eqref{eq formula for the spectral density}
gives the answer and completes the proof.
\end{proof}

\section{Appendix. Asymptotics of derivatives of the error function}
This section is devoted to finding asymptotics of $w^{(n)}(z)$ as
$\ninf$. It is natural to prove a little wider result: asymptotics
of $w^{(n-1)}(\mu\sqrt{2n})$ as $\ninf$ uniform with respect to
the parameter $\mu$ in some neighbourhood of the point $0$. Such
asymptotics (with the scaled parameter) are called asymptotics of
Plancherel-Rotach type, after \cite{Plancherel-Rotach}, where the
authors proved such asymptotics for Hermite polynomials. Let
\begin{equation*}
    \varphi(z):=z+\sqrt{z^2-1}
\end{equation*}
be the inverse Zoukowski function with the branch chosen such that
$\varphi(0)=i$.

\begin{thm}\label{thm asymptotics of w^n-1musqrt2n}
    There exist $\mu_0$ such that
    \begin{equation}\label{eq asymptotics of w^n-1musqrt2n}
        w^{(n-1)}(\mu\sqrt{2n})=
        \sqrt{\frac2n}^{\: n}\frac{(n-1)!(-1)^{n-1}}{\sqrt{\pi}\sqrt{1-\varphi^2(\mu)}}
        \frac{e^{-\frac n2(\varphi(\mu)-2\mu)^2}}{(\varphi(\mu))^{n-1}}
        \l(1+O\l(\frac1{\sqrt n}\r)\r)
    \end{equation}
    as $\ninf$ uniformly with respect to $|\mu|<\mu_0$.
\end{thm}

\begin{proof}
One has from \eqref{eq error function}:
\begin{equation*}
    w^{(n-1)}(\mu\sqrt{2n})=\frac{(-1)^n(n-1)!}{\pi i}\int_{\Gamma_{-\mu\sqrt{2n}}^+}\frac{e^{-\zeta^2}d\zeta}{(\zeta+\mu\sqrt{2n})^{n}},
\end{equation*}
for the contour $\Gamma_z^+$ see Figure \ref{fig gamma pm z}.
Taking $\zeta=(z-\mu)\sqrt{\frac n2}$, one obtains:
\begin{equation*}
    w^{(n-1)}(\mu\sqrt{2n})=(-1)^n\sqrt{\frac2n}^{\: n-1}\frac{(n-1)!}{\pi i}\int_{\Gamma_0^+}\frac{e^{-\frac n2(z-2\mu)^2}}{z^n}dz.
\end{equation*}
Let us denote
\begin{equation*}
    f(z,\mu):=-\frac{(z-2\mu)^2}2-\ln z.
\end{equation*}
This function has a critical point $z=\varphi(\mu)$ (the point where its
derivative with respect to $z$ turns to zero). Due to Taylor's expansion, for every $\mu$
\begin{equation*}
    f(z,\mu)=f(\varphi(\mu),\mu)+\frac{f''(\varphi(\mu),\mu)}2(z-\varphi(\mu))^2+O(z-\varphi(\mu))^3
\end{equation*}
as $z\rightarrow \varphi(\mu)$. Let us denote
\begin{equation*}
    \begin{array}{c}
    a(\mu):=\sqrt{\frac{-2}{f''(\varphi(\mu),\mu)}}=\sqrt{\frac{2\varphi^2(\mu)}{\varphi^2(\mu)-1}}, \\
    s:=\frac{z-\varphi(\mu)}{a(\mu)},\\
    h(s,\mu):=f(a(\mu)s+\varphi(\mu),\mu)-f(\varphi(\mu),\mu)
    \end{array}
\end{equation*}
and change the variable in the integral. Then one has to integrate
over the contour $\{s=\frac{z-\varphi(\mu)}{a(\mu)},
z\in\Gamma_0^+\}$, which can be transformed into the real line for
values of $\mu$ small enough (since $\varphi(\mu)\rightarrow i$
and $a(\mu)\rightarrow1$ as $\mu\rightarrow0$, so the point
$s=-\frac{\varphi(\mu)}{a(\mu)}\rightarrow-i$ corresponds to the
point $z=0$). One will have:
\begin{multline*}
    w^{(n-1)}(\mu\sqrt{2n})=(-1)^n\sqrt{\frac2n}^{\: n-1}\frac{(n-1)!}{\pi i}a(\mu)e^{nf(\varphi(\mu),\mu)}\int_{-\infty}^{+\infty}e^{nh(s,\mu)}ds
    \\
    =(-1)^{n-1}\sqrt{\frac2n}^{\: n-1}\frac{\sqrt 2(n-1)!}{\pi\sqrt{1-\varphi^2(\mu)}}\frac{e^{-\frac n2(\varphi(\mu)-2\mu)^2}}{(\varphi(\mu))^{n-1}}
    \int_{-\infty}^{+\infty}e^{nh(s,\mu)}ds.
\end{multline*}
It remains to prove the following lemma.

\begin{lem}\label{lem steepest descent estimates}
There exists $\mu_1$ such that
\begin{equation*}
    \int_{-\infty}^{+\infty}e^{nh(s,\mu)}ds=\sqrt{\frac{\pi}n}+O\l(\frac1n\r)\as\ninf
\end{equation*}
uniformly with respect to $|\mu|<\mu_1$.
\end{lem}

\begin{proof}
We divide the proof into three parts.

1. Let us see that
\begin{equation*}
    \int_{-n^{-3/8}}^{n^{-3/8}}e^{nh(s,\mu)}ds=\sqrt{\frac{\pi}n}+O\left(\frac1n\right)\as\ninf
\end{equation*}
uniformly with respect to $\mu$ in some neighbourhood of $0$. One
has:
\begin{equation*}
    h(s,\mu)=-\frac{(a(\mu)s)^2}2-a(\mu)s(\varphi(\mu)-2\mu)-\ln{\left(1+\frac{a(\mu)s}{\varphi(\mu)}\right)}.
\end{equation*}
Note that
\begin{equation*}
    h(0,\mu)\equiv0,\ h_s'(0,\mu)\equiv0,\
    h_{ss}''(0,\mu)\equiv-2.
\end{equation*}
Hence for every $k\ge0$
\begin{equation*}
    \frac{\partial^{k}h}{\partial\mu^{k}}(0,0)
    =\frac{\partial^{k+1}h}{\partial s\partial\mu^{k}}(0,0)
    =\frac{\partial^{k+3}h}{\partial s^2\partial\mu^{k+1}}(0,0)=0.
\end{equation*}
The function $h(s,\mu)$ is $C^{\infty}$ at $(0;0)$, so
\begin{equation*}
    h(s,\mu)=-s^2+O(s^3)\as s,\mu\rightarrow0
\end{equation*}
(i.e., there exist $C_1,\delta_1$ such that if
$|s|,|\mu|<\delta_1$, then $|h(s,\mu)+s^2|<C_1|s|^3$). This
obviously in particular means that there exists $\delta_0>0$ such
that if $-\delta_0<s<\delta_0$ and $|\mu|<\delta_0$, then
\begin{equation}\label{est estimates}
    \left\{
    \begin{array}{l}
      |h(s,\mu)+s^2|<C_1|s|^3\\
      \Re\ h(s,\mu)<-\frac{s^2}2.\\
    \end{array}
    \right.
\end{equation}
One has:
\begin{multline*}
    \int_{-n^{-3/8}}^{n^{-3/8}}e^{nh(s,\mu)}ds-\sqrt{\frac{\pi}n}
    \\
    =\int_{-n^{-3/8}}^{n^{-3/8}}(e^{nh(s,\mu)}-e^{-ns^2})ds-
    \left(\int_{-\infty}^{-n^{-3/8}}+\int_{n^{-3/8}}^{+\infty}\right)e^{-ns^2}ds.
\end{multline*}
Since for every $\alpha,\beta>0$,
\begin{equation}\label{est known estimate}
    \int_{x}^{+\infty}t^{\alpha}e^{-\beta t^2}dt=O(x^{\alpha+1}e^{-\beta x^2})\as x\rightarrow+\infty,
\end{equation}
one has:
\begin{equation*}
    \left(\int_{-\infty}^{-n^{-3/8}}+\int_{n^{-3/8}}^{+\infty}\right)e^{-ns^2}ds=O\l(\frac1n\r)\as\ninf.
\end{equation*}
Let $|s|<\min\{n^{-\frac38};\delta_0\}$ and $|\mu|<\delta_0$. Then
\begin{equation*}
    n|h(s,\mu)+s^2|<C_1n|s|^3<\frac{C_1}{n^{\frac18}}
\end{equation*}
and there exists $N_1$ such that
\begin{multline*}
\begin{array}{l}
  \text{if }n>N_1,|s|<n^{-3/8}\text{ and }|\mu|<\delta_0, \\
  \text{then }|e^{n(h(s,\mu)+s^2)}-1|<2C_1n|s|^3. \\
\end{array}
\end{multline*}
Hence we arrive at the following (uniform for $|\mu|<\delta_0$)
estimate:
\begin{multline*}
    \l|\int_{-n^{-3/8}}^{n^{-3/8}}(e^{nh(s,\mu)}-e^{-ns^2})ds\r|\le\int_{-n^{-3/8}}^{n^{-3/8}}e^{-ns^2}|e^{n(h(s,\mu)+s^2)}-1|ds
    \\
    <2C_1n\int_{-n^{-3/8}}^{n^{-3/8}}|s|^3e^{-ns^2}ds=O\l(\frac1n\r)\as\ninf
\end{multline*}
from \eqref{est known estimate}.

2. The following is an immediate consequence of \eqref{est
estimates} and \eqref{est known estimate}: for $|\mu|<\delta_0$,
\begin{multline*}
    \left|\left(\int_{-\delta_0}^{-n^{-3/8}}+\int_{n^{-3/8}}^{\delta_0} \right)
    e^{nh(s,\mu)}ds\right|<2\int_{n^{-3/8}}^{\delta_0}e^{-\frac{ns^2}2}ds
    \\
    <\frac2{\sqrt n}\int_{n^{1/8}}^{+\infty}e^{-\frac{t^2}2}dt
    =O\l(\frac1n\r)\as\ninf
\end{multline*}
uniformly with respect to $\mu$.

3. Let us prove that
\begin{equation*}
    \l(\int_{-\infty}^{-\delta_0}+\int_{\delta_0}^{+\infty}\r)e^{nh(s,\mu)}ds=O\left(\frac1n\right)\as\ninf
\end{equation*}
uniformly with respect to $\mu$ in some neighbourhood of $0$.
Consider the real part of the last term in
\begin{equation}\label{def h}
    h(s,\mu)=-\frac{(a(\mu)s)^2}2-a(\mu)s(\varphi(\mu)-2\mu)-\ln{\left(1+\frac{a(\mu)s}{\varphi(\mu)}\right)}.
\end{equation}
One has
\begin{equation*}
    \Re\ln{\left(1+\frac{a(\mu)s}{\varphi(\mu)}\right)}=\ln|i+\Gamma(\mu)s|,
\end{equation*}
\begin{figure}[h]
    \includegraphics[width=8 cm]{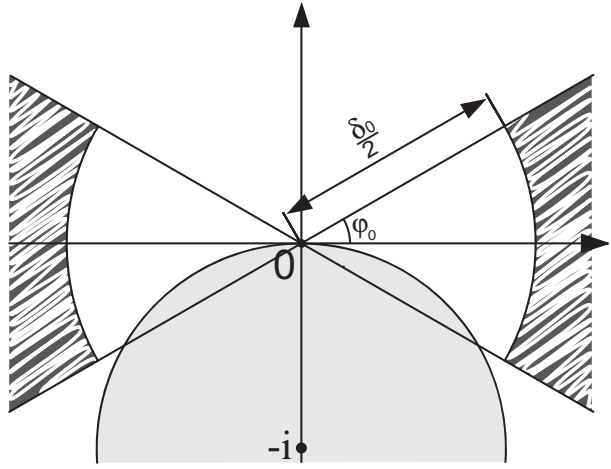}\\
    \caption{The plane of the parameter $s$}
    \label{fig circle}
\end{figure}
where
\begin{equation*}
    \gamma(\mu):=\frac{ia(\mu)}{\varphi(\mu)}.
\end{equation*}
Consider $s\in(-\infty;-\delta_0)\cup(\delta_0;+\infty)$. There
exists an angle $\varphi_0$ (small enough) such that the domains
shown on Figure \ref{fig circle} do not intersect. Since
$\gamma(\mu)\rightarrow1$ as $\mu\rightarrow0$, there exists
$\mu_1<\delta_0$ such that if $|\mu|<\mu_1$, then
$|\gamma(\mu)|>\frac12$ and $|\arg\gamma(\mu)|<\varphi_0$. Then
$|i+\gamma(\mu)s|>1$. Let $\theta:=\frac13$. By the choice of
$\mu_1$ we can also ensure that if $|\mu|<\mu_1$, then
\begin{equation*}
    \l\{
    \begin{array}{l}
        \Re\, a^2(\mu)>\frac12, \\
        \Re[a(\mu)(\varphi(\mu)-2\mu)]>-\frac{\delta_0\theta}2 \\
    \end{array}
    \r.
\end{equation*}
and hence
\begin{equation*}
    \Re h(s,\mu)<-\frac14(s^2-2s\delta_0\theta)
\end{equation*}
for every real $s$ such that $|s|>\delta_0$. One has:
\begin{multline*}
    \left|\left[\int_{-\infty}^{-\delta_0}+\int_{\delta_0}^{+\infty}\right]
    e^{nh(s,\mu)}ds\right|<2\int_{\delta_0}^{+\infty}e^{-\frac n4(s^2-2s\delta_0\theta)}ds
    \\
    =2e^{\frac n4\delta_0^2\theta^2}\int_{\delta_0(1-\theta)}^{+\infty}e^{-\frac n4s^2}ds
    =O\l(e^{\frac n4\delta_0^2(2\theta-1)}\r)=O\l(\frac1n\r)\as\ninf
\end{multline*}
uniformly with respect to $\mu$. This completes the proof of the
lemma.
\end{proof}

\end{proof}

As a corollary we have asymptotics of $w^{(n-1)}(z)$ as $\ninf$ for fixed $z$.

\begin{cor}\label{cor asymptotics of w^(n-1)(z)}
    \begin{equation*}
        w^{(n-1)}(z)=\sqrt{\frac2n}^{\: n}\frac{(n-1)!i^{n-1}}{\sqrt{2\pi}}e^{\frac n2+iz\sqrt{2n}-\frac{z^2}2}\l(1+O\l(\frac1{\sqrt n}\r)\r)\as\ninf
    \end{equation*}
    uniformly with respect to $z$ in every bounded set in $\mathbb C$.
\end{cor}

\begin{proof}
We just need to substitute $\mu=\frac z{\sqrt{2n}}$ into \eqref{eq
asymptotics of w^n-1musqrt2n} and go through tedious calculation,
using that
\begin{equation*}
    \varphi(z)=i+z-\frac{iz^2}2+O(z^4)\as z\rightarrow0.
\end{equation*}
\end{proof}

\section{Acknowledgement}
The author expresses his deep gratitude to Dr. A.V. Kiselev for
valuable discussions of the problem and to Prof. S.N. Naboko for
his constant attention to this work and for many fruitful
discussions of the subject. The author also wishes to thank Prof.
V.S. Buslaev for constructive criticism. The work was supported by
grants RFBR-06-01-00249 and INTAS-05-1000008-7883, and also by
Vladimir Deich prize.

\end{document}